\documentclass{amsart}

\usepackage[utf8]{inputenc}
\usepackage[T1]{fontenc}

\usepackage{mathpazo}
\usepackage{microtype}

\usepackage{amsmath, amsthm, amssymb}
\usepackage[all]{xy}

\newtheorem{theorem}{Theorem}[section]
\newtheorem{proposition}[theorem]{Proposition}
\newtheorem*{proposition*}{Proposition}
\theoremstyle{remark}
\newtheorem{claim}{Claim}
\newtheorem*{claim*}{Claim}

\numberwithin{equation}{section}

\newcommand{\f}{\mathfrak}
\newcommand{\K}{\mathbb{K}}
\newcommand{\F}{\mathbb{F}}
\newcommand{\R}{\mathbb{R}}
\newcommand{\C}{\mathbb{C}}
\newcommand{\im}{\mathbf{i}}

\newcommand{\set}[2]{\left\{ #1 \mid \allowbreak #2 \right\}}

\DeclareMathOperator{\Span}{Span}
\DeclareMathOperator{\Hom}{Hom}
\DeclareMathOperator{\Der}{Der}
\DeclareMathOperator{\ad}{ad}
\DeclareMathOperator{\dsum}{\dot{+}}

\newcommand{\rect}[3]{
  \multiput(#1)(#2,0){2}{\line(0,1){#3}}
  \multiput(#1)(0,#3){2}{\line(1,0){#2}}}

\begin{document}

\title{On derivations of parabolic {L}ie algebras}

\author{Daniel Brice}
\address{Department of Mathematics,
  California State University, Bakersfield, California 93311, USA}
\email{daniel.brice@csub.edu}

\date{\today}

\begin{abstract}
  Let $\mathfrak{g}$ be a reductive Lie algebra over an
  algebraically closed, characteristic zero field or over $\mathbb{R}$.
  Let $\mathfrak{q}$ be a parabolic subalgebra of $\mathfrak{g}$.
  We characterize the derivations of $\mathfrak{q}$ by decomposing the
  derivation algebra as the direct sum of two ideals: one of which being
  the image of the adjoint representation and the other consisting of
  all linear transformations on $\mathfrak{q}$ that map into the center
  of $\mathfrak{q}$ and map the derived algebra of $\mathfrak{q}$ to
  $0$.
\end{abstract}

\maketitle

\section{Introduction}\label{sec:introduction}

The study of derivations belongs to the classical theory of Lie
algebras.
We begin with the well-known result that if $\f g$ is a semisimple Lie
algebra over a field of characteristic not equal to two, then $\f g$
admits only inner derivations
\cite{humphreys1972introduction, strade1988modular},
in which case $\f g \cong \Der \f g$.
By 1972, Leger and Luks extended this result to the Borel algebras of
$\f g$ \cite{leger1972cohomology}.
More generally, their result applies to the class of Lie algebras $\f b$
that can be expressed as the semidirect product
$\f b = \f a \rtimes \f b'$ where the subalgebra $\f b'$ is nilpotent
and the ideal $\f a$ is abelian and acts diagonally on $\f b'$
\cite{leger1972cohomology}.
This wider class of Lie algebras includes Borel subalgebras of a
semisimple $\f g$ but does not include parabolic subalgebras.
Working independently, Tolpygo arrived at the same result for the
parabolic subalgebras $\f q$ of a semisimple algebra $\f g$, but only in
the special case that the scalar field is the complex numbers $\C$
\cite{tolpygo1972cohomologies}.

The recent direction that work on derivations has taken has been to
relax the definition of Lie algebra to include consideration of Lie
algebras that draw scalars from commutative rings rather than from
fields, characterizing derivations of specific classes of such Lie
algebras
\cite{ou2007derivations, wang2008derivations, wang2006derivations}.
Zhang in 2008 takes a different approach, defining a new class of
solvable Lie algebras over $\C$ and characterizing their derivation
algebras \cite{zhang2008class}.

Other work has been in the direction of considering certain maps that
are similar to but potentially fail to be derivations
\cite{chen2011nonlinear, chen2012nonlinear, wang2010product}.
Wang et al. recently defined a \emph{product zero derivation} of a Lie
algebra $\f g$ as a linear map $f: \f g \to \f g$ satisfying
$[f(x),y]+[x,f(y)]=0$ whenever $[x,y]=0$ \cite{wang2010product}.
The authors go on to characterize the product zero derivations of
parabolic subalgebras $\f q$ of simple Lie algebras over an
algebraically closed, characteristic zero field, ultimately showing all
product zero derivations of $\f q$ to be sums of inner derivations and
scalar multiplication maps \cite{wang2010product}.
In papers appearing in 2011 and 2012, Chen et al. consider nonlinear
maps satisfying derivability and nonlinear Lie triple derivations
\cite{chen2011nonlinear, chen2012nonlinear}.
The authors characterize all such maps on parabolic subalgebras of a
semisimple Lie algebra over $\C$ as the sums of inner derivations and
certain maps called quasi-derivations that may fail to be linear
\cite{chen2011nonlinear, chen2012nonlinear}.

The purpose of this paper is to extend the classical results of Leger
and Luks \cite{leger1972cohomology} and Tolpygo
\cite{tolpygo1972cohomologies} to the case where $\f q$ is a parabolic
subalgebra of a reductive algebra $\f g$.
We prove the following theorem.

\begin{proposition*}
  Let $\f q$ be a parabolic subalgebra of a reductive Lie algebra
  $\f g$ over an algebraically closed, characteristic zero field $\K$
  or over $\R$.
  Let $\f L$ be the set of all linear transformations mapping $\f q$
  into its center $\f q_Z$ and sending $[\f q, \f q]$ to $0$.
  Then $\f L$ is an ideal of $\Der \f q$ and $\Der \f q$ decomposes as
  the direct sum of ideals
  \[
    \Der \f q = \f L \oplus \ad \f q\text{.}
  \]
\end{proposition*}

The proof of the algebraically closed case is constructive: given a
derivation $D$ on $\f q$, we explicitly construct a linear map $L$ and
an element $x \in \f q$ such that $D = L + \ad x$, after which we prove
that our construction satisfies the stated properties.
The proof the real case appeals to the complex case as applied to
$\hat{\f q}$, the complexification of the real parabolic subalgebra
$\f q$.

That the method of proof of the real case relies on the complex case is
due to the following consideration.
A Lie algebra over an algebraically closed field support a more regular
structural decomposition then a real Lie algebra affords.
In particular, Langland's decomposition---an important tool in the proof
of the real case---is unnecessary for the algebraically closed case.

The method of proof for the algebraically closed relies on utilization
of the root system $\Phi$ of $\f g$.
In order to motivate the methods employed, we offer the following
example.
The reader is encouraged to keep this example in mind throughout the
sequel.

We consider the parabolic subalgebra $\f q$ of $\f g = \f{gl}_6(\C)$
consisting of block upper triangular matrices in block sizes 3, 2, 1
(see figure \ref{fig:parabolic321}).
We write $\f{gl}_6(\C) = \f g_Z \oplus \f g_S$, where the center
$\f g_Z = \C I$ and maximal semisimple ideal $\f g_S = \f{sl}_6(\C)$.
We decompose $\f q$ similarly: $\f q = \f g_Z \oplus \f q_S$,
where $\f q_S = \f q \cap \f g_S$.

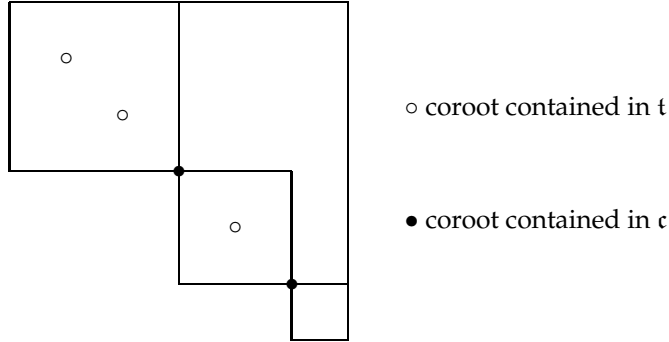
\begin{figure}[ht]
  \begin{center}
    \setlength{\unitlength}{0.75cm}
    \begin{picture}(12,6)
      \put(7,4){$\circ$ coroot contained in $\f t$}
      \put(7,2){$\bullet$ coroot contained in $\f c$}
      \put(.88,4.88){$\circ$}
      \put(1.88,3.88){$\circ$}
      \put(2.88,2.88){$\bullet$}
      \put(3.88,1.88){$\circ$}
      \put(4.88,0.88){$\bullet$}
      \put(3,6){\line(1,0){3}}
      \put(6,6){\line(0,-1){5}}
      \rect{0,3}{3}{3}
      \rect{3,1}{2}{2}
      \rect{5,0}{1}{1}
    \end{picture}
    \caption{Decomposition of $\f q_S$}
    \label{fig:parabolic321}
  \end{center}
\end{figure}

$\f g_S$ has root space decomposition
\[
  \f g_S = \f h \dsum \sum_{i \neq j}{\C e_{i,j}}
\]
where $\f h$ consists of traceless diagonal $6 \times 6$ matrices.
It is well known that the coroots $h_i = e_{ii} - e_{i+1,i+1}$ form a
basis of $\f h$.
We further decompose $\f h$ into $\f t \dsum \f c$,
where $\f t = \Span \{ h_1, h_2, h_4 \}$
and $\f c = \Span \{ h_3, h_5 \}$
(see figure \ref{fig:parabolic321}).
It follows that $\f t = \f h \cap [\f q, \f q]$ and that $\f q$ has the
vector space direct sum decomposition
\[
  \f q = \f g_Z \dsum \f c \dsum [\f q, \f q]\text{.}
\]
In light of this decomposition (and noting that $\f g_Z = \f q_Z$),
a linear transformation that sends $\f q$ to $\f q_Z$ and sends
$[\f q, \f q]$ to $0$ has the block matrix form illustrated by figure
\ref{fig:l_blockform}.

\begin{figure}[ht]
  \[
    \bordermatrix{
      &\hfill\f g_Z\hfill&\hfill\f c\hfill&\hfill[\f q, \f q]\hfill\cr
      \hfill\f g_Z\hfill&\ast&\ast&0\cr
      \hfill\f c\hfill&0&0&0\cr
      \hfill[\f q, \f q]\hfill&0&0&0
    }
  \]
  \caption{Block matrix form of derivations in $\f L$}
  \label{fig:l_blockform}
\end{figure}

The claim of the theorem---that $\Der \f q = \f L \oplus \ad \f q $---
may then be explicitly verified via computation in this example case.
The proofs of the theorem in general will rely on carrying out the same
decomposition of $\f q$ and the accompanying computations in abstract.

We give a brief outline of this paper.
Section \ref{sec:preliminaries} develops the necessary tools necessary
for the results in the sequel.
Section \ref{sec:complex_case} treats the case where $\f g$ is a Lie
algebra over an algebraically closed, characteristic zero field and
section \ref{sec:real_case} treats the case where $\f g$ is real.
Section \ref{sec:corollaries} contains several corollaries and a short
discussion of possible directions in which to expand upon the present
results.

Before we begin, we shall make note of some conventions of terminology
and notation.
If $\f g$ is a Lie algebra, we will denote its center by $\f g_Z$.
If $\f g$ is reductive, we denote its unique maximal semisimple ideal
by $\f g_S$.
If $\f g_1$ and $\f g_2$ are subspaces of $\f g$, intersect trivially,
and together span $\f g$, we write $\f g = \f g_1 \dsum \f g_2$.
If $\f a$ is a subalgebra (denoted $\f a \leq \f g$) and $\f b$ is an
ideal (denoted $\f b \unlhd \f g$), we will write
$\f g = \f a \ltimes \f b$ or $\f g = \f b \rtimes \f a$
interchangeably.
The notation $\f g = \f a \oplus \f b$ is reserved for the special case
where both $\f a$ and $\f b$ are ideals of $\f g$.

\section{Preliminaries}\label{sec:preliminaries}

This section develops the basic facts of the classical theory of Lie
algebras which are required for an understanding of the subsequent
discussion.
Where uncredited, the propositions in this section are elementary,
and as such are assumed to be well-known;
however, proofs are included since specific references are not
mentioned.

\begin{theorem}[Ado's Theorem]\label{thm:ados_thm}
  Let $\F$ be a characteristic zero field.
  Let $\f g$ be a (finite-dimensional) Lie algebra over $\F$.
  Then, $\f g$ is isomorphic to a space of matrices with entries in $\F$
  and bracket $[M,N] = MN-NM$ \cite[Ch. I, \S 7.3]{bourbaki1975lie}.
\end{theorem}

\begin{proposition}\label{prop:inner_derivations}
  An inner derivation maps $\f g$ into $[\f g, \f g]$ and stabilizes
  ideals.
\end{proposition}

\begin{proof}
  Let $D$ be an inner derivation,
  so $D = \ad x$ for some $x \in \f g$.
  Let $y \in \f g$ be arbitrary and notice
  $D(y) = \ad x (y) = [x,y] \in [\f g,\f g]$,
  verifying the first assertion.
  Next, let $\f a \unlhd \f g$.
  $D(\f a) = \ad x (\f a) = [x,\f a] \subseteq \f a$
  by the definition of ideal.
\end{proof}

An outer derivation does not necessarily stabilize ideals;
however, the derived algebra and center of $\f g$ are stabilized by
outer derivation.

\begin{proposition}\label{prop:derivations_stabilize_these}
  Let $D$ be a derivation on an arbitrary Lie algebra $\f g$.
  $D$ stabilizes $[\f g, \f g]$ and $\f g_Z$.
\end{proposition}

\begin{proof}
  Let $x,y \in \f g$.
  \[
    D([x,y]) = [D(x),y]+[x,D(y)] \in [\f g, \f g],
  \]
  so $D$ stabilizes $[\f g,\f g]$ as desired.

  Next, let $z \in \f g_Z$.
  We want $D(z) \in \f g_Z$.
  Let $x \in \f g$ and consider $D([z,x]) = 0$.
  \[
    0 = D([z,x]) = [D(z),x] + [z,D(x)] = [D(z),x],
  \]
  so $[D(z),x] = 0$ for all $x \in \f g$, as desired.
\end{proof}

Let $\im$ denote the imaginary unit.
Let $\f g$ be a Lie algebra over $\R$.
In light of Ado's Theorem (proposition \ref{thm:ados_thm}),
$\f g$ is isomorphic to a real Lie algebra consisting of matrices with
real entries, and we think of $\f g$ in this way as we proceed in order
to avoid several issues with notation.
$\hat{\f g}$ will denote the complexification of $\f g$.
Since $\f g$ is real, we have
\[
  \hat{\f g}
  = \f g \dsum \im \f g
  = \set{x + \im y}{x,y \in \f g}
  \text{.}
\]
We note that the bracket on $\hat{\f g}$ is given by
\[
  [x+\im y, u+\im v] = [x,u] - [y,v] + \im ([x,v] + [y,u])\text{.}
\]

\begin{proposition}\label{prop:center_of_complexification}
  Let $\f g$ be real, let $\hat{\f g} = \f g \dsum \im \f g$ be the
  complexification of $\f g$.
  Then the center of $\hat{\f g}$ is the complexification of the center
  of $\f g$, namely
  $\hat{\f g}_Z = \widehat{\f g_Z} = \f g_Z \dsum \im \f g_Z$.
\end{proposition}

\begin{proof}
  Let $z \in \hat{\f g}_Z$. Write $z = x + \im y$ with $x, y \in \f g$.
  Now, for arbitrary $w = u + \im v \in \hat{\f g}$ with $u, v \in \f g$
  we have
  \begin{align*}
    0 &= [z,w] = [x + \im y, u + \im v]
    \\ &= [x,u] - [y,v] + \im ([x,v] + [y,u])
  \end{align*}
  and by direct sum decomposition $[x,u] = [y,v]$ and $[x,v] = - [y,u]$.
  Adding these equations gives
  \begin{equation}\label{eq:funny_sum}
    \forall u,v \in \f g, \quad [x,u+v] = [y,v-u]
  \end{equation}

  Setting $v = u$ in \eqref{eq:funny_sum} produces $[x,2u]=0$
  for all $u \in \f g$, so $x \in \f g_Z$.
  Similarly, setting $u = -v$ in \eqref{eq:funny_sum} produces
  $0=[y,2v]$ for all $v \in \f g$, so $y \in \f g_Z$, giving
  $\hat{\f g}_Z \subseteq \widehat{\f g_Z}$.
  The reverse inclusion is clear.
\end{proof}

\begin{theorem}
  Let $\f g$ be a semisimple (res. reductive) Lie algebra over $\R$.
  The complexification $\hat{\f g}$ of $\f g$ is semisimple
  (res. reductive) \cite[Ch. VI, \S 9]{knapp2002lie}.
\end{theorem}

\begin{proposition}
  Let $D$ be a derivation of the real Lie algebra $\f g$.
  Then $\hat D$ defined by $\hat D (x + \im y) = D(x) + \im D(y)$
  is a derivation of $\hat{\f g}$.
  Moreover, $\hat D$ stabilizes $\f g$.
\end{proposition}

\begin{proof}
  Let $z = x + \im y, w = u + \im v$ be arbitrary elements of
  $\hat{\f g}$.
  \begin{align*}
    \hat D ([z,w]) &= \hat D ( [x+\im y,u+\im v] )
    \\ &= \hat D \big( [x,u]-[y,v] + \im ( [x,v]+[y,u] ) \big)
    \\ &= D([x,u]-[y,v]) + \im D ([x,v]+[y,u])
    \\ &= D([x,u]) - D([y,v]) + \im ( D([x,v]) + D([y,u]) )
    \\ &= [D(x),u] + [x,D(u)] - [D(y),v] - [y,D(v)]
    \\ &\quad+ \im ( [D(x),v] + [x,D(v)] + [D(y),u] + [y,D(u)] )
    \\ &= [D(x),u + \im v] + [x + \im y,D(u)]
    + \im [x + \im y, D(v)] + \im [ D(y),u + \im v]
    \\ &= [D(x), w] + \im [D(y), w] + [z,D(u)] + \im[z,D(v)]
    \\ &= [D(x) + \im D(y), w] + [z, D(u) + \im D(v)]
    \\ &= \left[\hat D (z),w\right] + \left[z,\hat D (w)\right]
  \end{align*}
  So $\hat D$ is a derivation on $\hat{\f g}$.

  $\hat D$ stabilizes $\f g$ by definition.
  Indeed, if $x \in \f g$, then
  \[
    \hat D (x) = \hat D (x + \im 0) = D(x) \in \f g\text{.}\qedhere
  \]
\end{proof}

\begin{proposition}
  Let $\f q = \f g_Z \oplus \f q_S$ be a parabolic subalgebra
  of the reductive Lie algebra $\f g = \f g_Z \oplus \f g_S$
  over $\K$ or over $\R$.
  The center of $\f q$ is $\f g_Z$.
\end{proposition}

\begin{proof}
  We consider first the special case where
  $\f g = \f h \dsum \sum_{\beta \in \Phi}$ is semisimple over $\K$.
  We assume without loss of generality that $\f q$ is a standard
  parabolic subalgebra,
  and $\f b=\f h\dsum\sum_{\beta\in\Phi^+} \f g_{\beta}$ is the Borel
  subalgebra of $\f g$ such that $\f b\subseteq\f q$.
  Then
  \[
    {\f q}_Z =
    Z_{\f q}(\f q)\subseteq Z_{\f g}(\f q)\subseteq Z_{\f g}(\f b)
    \text{.}
  \]
  We prove that $Z_{\f g}(\f b)=0$, so that ${\f q}_Z=0$.
  Choose $x_{\alpha}\in{\f g}_{\alpha}\setminus\{0\}$ for every
  $\alpha\in\Phi$.
  Let
  \[
    z=z_{\f h}+\sum_{\alpha\in\Phi}c_{\alpha}x_{\alpha}
  \]
  be an arbitrary element of  $Z_{\f g}(\f b)$,
  where $z_{\f h}\in\f h$, $c_{\alpha}\in \K$.
  \begin{enumerate}
    \item
    For every $\beta\in\Phi^+$, $x_{\beta}\in \f b$ so that
    \[
      0=[z,x_{\beta}]
      =[z_{\f h}, x_{\beta}]
      +\sum_{\alpha\in\Phi} c_{\alpha}[x_{\alpha},x_{\beta}]\
      \in\ \beta(z_{\f h}) x_{\beta}
      +\left(
      \f h+\sum\set{\f g_{\alpha}}{\alpha\in\Phi\setminus\{\beta\}}
      \right)
      \text{.}
    \]
    Therefore, $\beta(z_{\f h})=0$ for every $\beta\in\Phi^+$,
    so that $z_{\f h}=0$.
    \item
    Every $h\in\f h$ is also in $\f b$. So
    \[
      0=[z,h]
      =\sum_{\alpha\in\Phi}c_{\alpha}[x_{\alpha}, h]
      =-\sum_{\alpha\in\Phi}\alpha(h)c_{\alpha}x_{\alpha}
      \text{.}
    \]
    For every $\alpha\in\Phi$, we may choose $h\in\f h$ such that
    $\alpha(h)\ne 0$.
    Then $c_{\alpha}=0$.
  \end{enumerate}

  The above argument shows that $z=0$.
  Therefore, $\f q_Z=Z_{\f g}(\f b)=0$.

  Having established that $\f q_Z = 0$ when $\f g$ is semisimple over
  $\K$,
  that $\f q_Z = \f g_Z$ when $\f g$ is reductive over $\K$ follows from
  the Lie algebra direct sum decomposition
  $\f q = \f g_Z \oplus \f q_S$.
  We now consider the case where $\f q$ is a parabolic subalgebra of a
  real reductive $\f g$.
  We have $\hat{\f q}$ is a parabolic subalgebra of $\hat{\f g}$
  by definition.
  Then
  \begin{equation}\label{eq:center_of_real_parabolic}
    \f g_Z + \im \f g_Z = \widehat{(\f g_Z)}
    = \underbrace{
      (\hat{\f g})_Z = (\hat{\f q})_Z
    }_{\text{by above case}}
    = \widehat{(\f q_Z)} = \f q_Z + \im \f q_Z
    \text{.}
  \end{equation}
  Finally, by Ado's Theorem (theorem \ref{thm:ados_thm}), we may assume
  that $\f g$ consists of real matrices, so that we may separate the
  real and imaginary part in \eqref{eq:center_of_real_parabolic},
  giving $\f g_Z = \f q_Z$, as desired.
\end{proof}

What follows of this section is developed more completely in chapter V,
section 7 of \cite{knapp2002lie} in case $\f g$ is over $\K$
and in chapter VII,
section 7 of \cite{knapp2002lie} in case $\f g$ is over $\R$.

Let $\f g$ be semisimple over $\K$ or over $\R$ and let $\f q \leq \f g$
be a parabolic subalgebra.
Without loss of generality, we may assume that $\f q$ arises as a
standard parabolic subalgebra from a (restricted) root space
decomposition of $\f g$.

In the algebraically closed case, we have the following situation:
\begin{itemize}
  \item[] $\f g = \f h \dsum \sum_{\beta \in \Phi}{\f g_\beta}$,
  where
  \item[] $\f h$ is a Cartan subalgebra of $\f g$,
  \item[] $\Phi$ is the root system of $\f g$ relative to $\f h$,
  \item[] $\Delta$ is a base of $\Phi$,
  \item[] $\Delta' \subseteq \Delta$ is the subset of $\Delta$
  corresponding to $\f q$, and
  \item[] $\Phi' = \Phi^+ \cup \left( \Phi\cap \Span \Delta'\right)$.
\end{itemize}
Then $\f q = \f h \dsum \sum_{\beta \in \Phi'}{\f g_{\beta}}$.

Considering the case where $\f g$ is real, we have the analogous
situation:
\begin{itemize}
  \item[] $\f g = \f a \dsum \f m \dsum
  \sum_{\lambda \in \Phi}{\f g_\lambda}$ where,
  \item[] $\f g = \f k \dsum \f p$ is the Cartan decomposition of
  $\f g$,
  \item[] $\f a$ is a maximal abelian subspace of $\f p$,
  \item[] ${\f m}=Z_{\f k}({\f a})$ is the centralizer of $\f a$ in
  $\f k$,
  \item[] $\Phi$ is the restricted root system of $\f g$ relative to
  $\f a$,
  \item[] $\Delta$ is a set of simple restricted roots of $\Phi$,
  \item[] $\Delta' \subseteq \Delta$ is the subset of $\Delta$
  corresponding to $\f q$, and
  \item[] $\Phi' = \Phi^+ \cup \left( \Phi\cap \Span \Delta'\right)$,
\end{itemize}
so that $\f q = \f a \dsum \f m
\dsum\sum\set{\f g_{\lambda}}{\lambda \in \Phi'}$.

$\Phi'$ may be partitioned into two subsets, $\Phi' \cap -\Phi'$ and
$\Phi' \setminus -\Phi'$. This partition of $\Phi'$ results in a vector
space direct sum decomposition of $\f q$ as $\f q = \f l \dsum \f n$,
where
\[
  \f l = \f h \dsum \sum \set{\f g_\beta}{\beta \in \Phi' \cap -\Phi'}
\]
in case $\f g$ is over an algebraically closed field, or
\[
  \f l = \f a \dsum \f m \dsum
  \sum \set{\f g_\beta}{\beta \in \Phi' \cap -\Phi'}
\]
in case $\f g$ is over $\R$, and
\[
  \f n = \sum\set{\f g_\beta}{\beta \in \Phi' \setminus -\Phi'}\text{.}
\]

\begin{theorem}
  Notation as above, $\f n$ is an ideal of $\f q$, $\f l$ is a
  subalgebra of $\f q$, and $\f l$ is reductive
  \cite[Ch. V, \S 7]{knapp2002lie} \cite[Ch. VII, \S 7]{knapp2002lie}.
\end{theorem}

$\f l$ is called the \emph{Levi factor} of $\f q$ and $\f n$ is called
the \emph{nilradical} of $\f q$.
The decomposition $\f q = \f l \ltimes \f n$ is referred to as
\emph{Langland's decomposition}.

We extend this terminology and notion to the case where
$\f g = \f g_Z \oplus \f g_S$ is reductive and
$\f q = \f g_Z \oplus \f q_S$, by simply writing
\[
  \f q = \f g_Z \oplus \left( \f l \dsum \f n \right)
\]
where $\f l$ is the Levi factor and $\f n$ is the nilpotent radical of
$\f q_S$.
In such case, we say $\f l$ (res. $\f n$) is the Levi factor
(res. nilradical) of $\f q$ and of $\f q_S$ interchangeably.

\section{The algebraically closed case}\label{sec:complex_case}

Throughout this section we use the following notational conventions:
\begin{itemize}
  \item[]
  $\K$ denotes an algebraically closed, characteristic zero field;
  \item[]
  $\f g = \f g_Z \oplus \f g_S$ denotes a reductive Lie algebra over
  $\K$, where
  \item[]
  $\f g_Z$ is the center of $\f g$, and
  \item[]
  $\f g_S$ is the maximal semisimple ideal of $\f g$;
  \item[]
  $\f q = \f g_Z \oplus \f q_S$ is a given parabolic subalgebra of
  $\f g$, where
  \item[]
  $\f q_S = \f q \cap \f g_S$ is a parabolic subalgebra of $\f g_S$.
\end{itemize}
We choose a Cartan subalgebra $\f h$, a root system $\Phi$,
and a base $\Delta$ compatible with $\f q_S$ in the sense that
$\f q_S$ is a standard parabolic subalgebra of $\f g_S$ relative to
$(\f h, \Phi, \Delta)$ and corresponds to a subset
$\Delta' \subseteq \Delta$.
Then
\[
  \f q_S=\f h \dsum \sum_{\alpha\in\Phi'}\K x_{\alpha}
\]
where
\[
  \Phi' = \Phi^+ \cup \left( \Phi\cap \Span \Delta'\right)
\]
and where each $x_\alpha$ is chosen arbitrarily from the one-dimensional
root space it spans.

Define $\f t$ and $\f c$ by
\begin{itemize}
  \item[]
  $\f  t = \f h \cap [\f q, \f q]$ and
  \item[]
  $\f c = \Span \set{[x_\alpha,x_{-\alpha}]}%
  {\alpha \in \Delta \setminus \Delta'}$.
\end{itemize}

\begin{claim*}
  $\f h$ decomposes as $\f h = \f c \dsum \f t$.
\end{claim*}

\begin{proof}
  Notice that
  \[
    \f h = \Span \set{[x_\alpha,x_{-\alpha}]}{\alpha \in \Delta}
  \]
  and that
  \[
    \f t = \Span \set{[x_\alpha,x_{-\alpha}]}{\alpha \in \Delta'}
    \text{.}
  \]
  From these observations,
  we see that $\f c \cap \f t = 0$ and that
  $\Span (\f c \cup \f t) = \f h$.
\end{proof}

Noting that
$[\f q, \f q] = \f t \dsum \sum_{\alpha \in \Phi'} \K x_\alpha$,
we arrive at the desired vector space direct-sum decompositions of
$\f q$:
\begin{align*}
  \f q
  &= \f g_Z \dsum \f h \dsum \sum_{\alpha \in \Phi'} \K x_\alpha
  \\
  &= \f g_Z \dsum
  \lefteqn{
    \overbrace{\phantom{\f c \dsum \f t}}^{\f h}
  }
  \f c \dsum
  \underbrace{
    \f t \dsum \sum_{\alpha \in \Phi'} \K x_\alpha
  }_{[\f q, \f q]}
  \\
  &= \f g_Z \dsum \f c \dsum [\f q, \f q]\text{.}
\end{align*}

We take a moment to note that alternatively $\f c$,
as a direct sum complement of $[\f q,\f q]$ in $\f q_S$,
may have been chosen so that it coincides with the center of $\f l$
in Langland's decomposition $\f q_S = \f l \dsum \f n$.
This approach is not required in the algebraically closed case,
but it is taken in order to simplify the proof of proposition
\ref{prop:real_decomp} of the real case.

For the remainder of the section, we assume all of the notational
conventions described above without further mention, starting with a
restatement of the central theorem in terms of the adopted notation.

\begin{proposition}\label{prop:complex_case}
  For a parabolic subalgebra $\f q=\f g_Z \oplus \f q_S$
  of a reductive Lie algebra $\f g=\f g_Z\oplus\f g_S$ over $\K$,
  the derivation algebra $\Der \f q$ decomposes as the direct sum
  of ideals
  \[
    \Der \f q = \f L \oplus \ad \f q\text{,}
  \]
  where $\f L$ consists of all $\K$-linear transformations
  on $\f q$ mapping into $\f q_Z$ and mapping $[\f q, \f q]$ to $0$.

  Explicitly, for any root system $\Phi$ with respect to which
  $\f q$ is a standard parabolic subalgebra,
  $\f q$ decomposes as $\f q=\f g_Z \dsum \f c \dsum [\f q,\f q]$,
  and the ideal $\f L$ consists of all $\K$-linear transformations
  on $\f q$ that map $\f g_Z+\f c$ into $\f g_Z$
  and map $[\f q,\f q]$ to $0$,
  whereby
  \[
    \Der \f q \cong
    \Hom_{\K} \left( \f g_Z \dsum \f c, \f g_Z \right) \oplus \f q_S
  \]
  as Lie algebras.
\end{proposition}

We must explain what is meant by
$\Hom_{\K} \left( \f g_Z \dsum \f c, \f g_Z \right)$ as a Lie algebra,
since it is merely a space of linear maps and does not come equipped
with a Lie bracket a priori.
For vector spaces $V_1, V_2$,
we consider the space $\Hom_{\K}(V_2, V_1)$ an abelian Lie algebra.
Then, $\Hom_{\K} \left( V_1 \dsum V_2, V_1 \right)$ may be realized as
the Lie algebra semidirect sum
\[
  \Hom_{\K} \left( V_1 \dsum V_2, V_1 \right)
  = \f{gl}(V_1) \ltimes \Hom_{\K}(V_2, V_1)
\]
with the action of $\f{gl}(V_1)$ on $\Hom_{\K}(V_2, V_1)$ defined by
\[
  f \cdot g =
  f \circ g \quad \forall f \in \f{gl}(V_1), g \in \Hom_{\K}(V_2, V_1)
  \text{.}
\]

This definition is intrinsic in the sense that
if we fix bases for $V_1$ and $V_2$,
then we may identify $\Hom_{\K} \left( V_1 \dsum V_2, V_1 \right)$
with the subalgebra of $\f{gl}(V_1 \dsum V_2)$
consisting of block matrices of the form illustrated
in figure \ref{fig:l_smallform}
(compare to figure \ref{fig:l_blockform}),
and the Lie bracket defined by the action above coincides
with the standard Lie bracket on matrices, $[M,N] = MN-NM$.

\begin{figure}[ht]
  \[
    \bordermatrix{
          &V_1  &V_2  \cr
      V_1 &\ast &\ast \cr
      V_2 &0    &0
    }
  \]
  \caption{
    Embedding
    of $\Hom_{\K} \left( V_1 \dsum V_2, V_1 \right)$
    in $\f{gl}(V_1 \dsum V_2)$
  }
  \label{fig:l_smallform}
\end{figure}

\begin{proof}[Proof of proposition \ref{prop:complex_case}]
  For clarity, the proof of the theorem is organized into a progression
  of claims.
  The first three claims establish that an arbitrary derivation may be
  written as a sum of an inner derivation and a derivation mapping
  $\f g_Z \dsum \f c$ to $\f g_Z$ and $[\f q, \f q]$ to $0$.
  To this end, let $D$ be an arbitrary derivation of $\f q$.

  \begin{claim}\label{claim:d_adx}
    There is an $x \in \f q$ such that $D - \ad x$
    maps $\f c$ to $\f g_Z$, annihilates $\f t$,
    and stabilizes each root space $\K x_\alpha$.
  \end{claim}

  Let $h, k \in \f h$ be arbitrary and write
  $D(h) = z + h' + \sum_\gamma a_\gamma(h) x_\gamma$
  and $D(k) = c + k' + \sum_\gamma a_\gamma(k) x_\gamma$
  with $z,c \in \f g_Z$ and $h',k' \in \f h$
  and $a_\gamma(h), a_\gamma(k) \in \K$.
  Recall $[h,k]=0$ since $h,k \in \f h$ and consider $D([h,k])$.
  \begin{align*}
    0 &= D([h,k])
    \\
    &=[h,D(k)]-[k,D(h)]
    \\
    &= \left[ h, c + k' + \sum_\gamma a_\gamma(k) x_\gamma \right]
    - \left[ k, z + h' + \sum_\gamma a_\gamma(h) x_\gamma \right]
    \\
    &= \left[ h, \sum_\gamma a_\gamma(k) x_\gamma \right]
    - \left[ k, \sum_\gamma a_\gamma(h) x_\gamma \right]
    \\
    &= \sum_\gamma a_\gamma(k) [h,x_\gamma]
    - \sum_\gamma a_\gamma(h) [k,x_\gamma]
    \\
    &= \sum_\gamma \left( a_\gamma(k) \gamma(h)
    - a_\gamma(h) \gamma(k) \right) x_\gamma
    \text{.}
  \end{align*}
  So
  \begin{equation}\label{eq:scalar_proj}
    a_\gamma(k) \gamma(h) - a_\gamma(h) \gamma(k) = 0
    \text{ for all } \gamma \in \Phi', h, k \in \f h
    \text{.}
  \end{equation}

  Furthermore, for any pair $h,k$ for which $\gamma(h) \neq 0$ and
  $\gamma(k) \neq 0$, we have that
  \[
    \frac{a_\gamma(h)}{\gamma(h)} = \frac{a_\gamma(k)}{\gamma(k)}
    \text{.}
  \]
  This observation, along with the fact that $\gamma(h) \neq 0$
  for at least one $h \in \f h$,
  allows us to associate with each $\gamma \in \Phi'$
  the numerical invariant
  \[
    d_\gamma = \frac{a_\gamma(h)}{\gamma(h)}
  \]
  independently of our choice of $h$.

  Notice that $a_\gamma(h) - d_\gamma \gamma(h) = 0$ by definition
  when $\gamma(h) \neq 0$.
  If $\gamma(h) = 0$, the same equality still holds,
  as \eqref{eq:scalar_proj} becomes
  \[
    a_\gamma(h) \gamma(k) = 0 \text{ for all } k \in \f h\text{.}
  \]
  Since at least one $k \in \f h$ satisfies $\gamma(k) \neq 0$ we have
  $a_\gamma(h) = 0$ in case $\gamma(h) = 0$, giving
  \begin{equation}\label{eq:d_gamma}
    a_\gamma(h) - d_\gamma \gamma(h) = 0 \text{ for all } h \in \f h
    \text{.}
  \end{equation}

  Now, set $x = \sum_\gamma - d_\gamma x_\gamma$.
  Write $D' = D - \ad x$.
  We will show that $D'$ maps  $\f c$ to $\f g_Z$, annihilates $\f t$,
  and stabilizes each root space $\K x_\alpha$.

  We first show that $D'$ maps $\f h$ to $\f g_Z \dsum \f h$.
  Let $h \in \f h$ be arbitrary and again write
  $D(h) = z + h' + \sum_\gamma a_\gamma(h) x_\gamma$.
  We have that
  \begin{align*}
    D'(h) &= D(h) - \ad x (h)
    \\
    &= z + h' + \sum_\gamma a_\gamma(h) x_\gamma
    -  \sum_\gamma - d_\gamma  \ad  x_\gamma  (h)
    \\
    &= z + h' + \sum_\gamma a_\gamma(h) x_\gamma
    - \sum_\gamma d_\gamma  \ad h (x_\gamma)
    \\
    &= z + h' + \sum_\gamma a_\gamma(h) x_\gamma
    - \sum_\gamma d_\gamma \gamma(h) x_\gamma
    \\
    &= z + h' + \sum_\gamma \underbrace{
      \left( a_\gamma(h) -  d_\gamma  \gamma(h) \right)
    }_{0 \text{ by \eqref{eq:d_gamma}}} x_\gamma
    \\
    &= z + h'
  \end{align*}
  affirming the assertion.

  Having established that $D'$ maps $\f h$ into $\f g_Z \dsum \f h$,
  we have left to show that $D'$ annihilates $\f t$ and stabilizes each
  $\K x_\alpha$.
  Let $h \in \f h$ and $\alpha \in \Phi'$ be arbitrary,
  and write $D' (h) = z + h'$ and $D' (x_\alpha) = c + k +
  \sum_\gamma b_\gamma x_\gamma$ with $z,c \in \f g_Z$ and
  $h', k \in \f h$ and $b_\gamma \in \K$.
  Consider $D' \left( [h,x_\alpha] \right)$.
  On one hand,
  \begin{align*}
    D' \left( [h,x_\alpha] \right)
    &= D' \left( \alpha(h) x_\alpha \right)
    \\
    &= \alpha(h)D'(x_\alpha)
    \\
    &= \alpha(h)c + \alpha(h)k
    + \sum_\gamma \alpha(h) b_\gamma x_\gamma
    \text{.}
    \stepcounter{equation}\tag{\theequation}\label{eq:sharp}
  \end{align*}
  On the other hand,
  \begin{align*}
    D' \left( [h,x_\alpha] \right)
    &= [D'(h),x_\alpha] + [h,D'(x_\alpha)]
    \\
    &= [z + h',x_\alpha] +
    \left[ h, c + k + \sum_\gamma b_\gamma x_\gamma \right]
    \\
    &= [h',x_\alpha] + \sum_\gamma b_\gamma[h,x_\gamma]
    \\
    &= \alpha(h')x_\alpha + \sum_\gamma \gamma(h) b_\gamma x_\gamma
    \\
    &= \left( \alpha(h') + \alpha(h)b_\alpha \right)x_\alpha
    + \sum_{\gamma \neq \alpha} \gamma(h) b_\gamma x_\gamma
    \text{.}
    \stepcounter{equation}\tag{\theequation}\label{eq:flat}
  \end{align*}
  By equating \eqref{eq:sharp} and \eqref{eq:flat} and by direct sum
  decomposition of $\f q$ we obtain
  \begin{align}
    \alpha(h)c &= 0,\label{eq:first1}\\
    \alpha(h)k &= 0,\label{eq:first2}\\
    \alpha(h) b_\gamma &= \gamma(h) b_\gamma
    \quad\text{for } \gamma \neq \alpha \text{, and}
    \label{eq:middle}\\
    \alpha(h) b_\alpha &= \alpha(h') + \alpha(h) b_\alpha
    \text{.}
    \label{eq:last}
  \end{align}

  Since $h$ is arbitrary, \eqref{eq:first1} and \eqref{eq:first2}
  give $c = 0$ and $k = 0$ respectively.
  Second, \eqref{eq:middle} give us
  $b_\gamma(\gamma - \alpha) (h) =0$
  for all $\gamma \neq \alpha$.
  If any one $b_\gamma \neq 0$, then we would have $\gamma = \alpha$,
  a contradiction, so each $b_\gamma = 0$,
  whence $D'$ stabilizes each root space.

  Next, \eqref{eq:last} gives us $0=\alpha(h')$.
  Since $\alpha$ is arbitrary in $\Phi'$ and $\Phi'$ contains a basis of
  $\f h^*$, $h' = 0$, so $D'(\f h) \subseteq \f g_Z$.
  Since derivations in general stabilize $[\f q,\f q]$,
  $D'(\f t) \subseteq \f g_Z \cap [\f q,\f q] = 0$,
  so $D'$ annihilates $\f t$.
  The claim is verified.

  \begin{claim}\label{claim:dadxadh}
    There is an $h \in \f h$ whereby $D - \ad x - \ad h$
    annihilates $[\f q,\f q]$.
  \end{claim}

  We have the $D' = D - \ad x$ maps $\f c$ to $\f g_Z$,
  annihilate $\f t$, and stabilize each root space $\K x_{\alpha}$.
  For each $\gamma \in \Phi'$ write
  \[
    D'(x_\gamma) = c_\gamma x_\gamma
  \]
  with $c_\gamma \in \K$.
  Taking each $\alpha \in \Delta$, the scalars $c_\alpha$ define a
  linear functional
  \[
    \tilde{c} : \f h^* \to \K.
  \]
  We first wish to show that for each $\gamma \in \Phi'$,
  $c_\gamma = \tilde{c}(\gamma)$.

  We begin with $\gamma \in \Phi' \cap \Phi^+$.
  We choose the expression $\gamma = \alpha_1 + ... + \alpha_k$
  with each $\alpha_i \in \Delta$ so that each sequential partial sum
  $\alpha_1 + ... + \alpha_i \in \Phi'$
  \cite[\S 10.2]{humphreys1972introduction}.
  Then
  \[
    ax_\gamma = \left[ \cdots
    \left[ \left[ x_{\alpha_1},
    x_{\alpha_2} \right],
    x_{\alpha_3} \right],
    \cdots,
    x_{\alpha_k} \right]
  \]
  for some $ 0 \neq a \in \K$.
  Apply $D'$ to both sides.
  By Leibniz rule, we have
  \begin{align*}
    c_\gamma a x_\gamma
    &= D' \left[
      \cdots \left[
        \left[
          x_{\alpha_1}, x_{\alpha_2}
        \right], x_{\alpha_3}
      \right], \cdots, x_{\alpha_k}
    \right]
    \\
    &= \left[
      \cdots \left[
        \left[
          D'(x_{\alpha_1}), x_{\alpha_2}
        \right],
        x_{\alpha_3}
      \right], \cdots , x_{\alpha_k}
    \right]
    \\
    &\quad+ \left[
      \cdots \left[
        \left[
          x_{\alpha_1}, D'(x_{\alpha_2})
        \right],
        x_{\alpha_3}
      \right], \cdots , x_{\alpha_k}
    \right]
    \\
    &\quad+ \left[
      \cdots \left[
        \left[
          x_{\alpha_1}, x_{\alpha_2}
        \right],
        D'(x_{\alpha_3})
      \right], \cdots , x_{\alpha_k}
    \right]
    \\
    &\quad+ ... + \left[
      \cdots \left[
        \left[
          x_{\alpha_1}, x_{\alpha_2}
        \right], x_{\alpha_3}
      \right], \cdots , D'(x_{\alpha_k})
    \right]
    \\
    &= c_{\alpha_1} \left[
      \cdots \left[
        \left[
          x_{\alpha_1}, x_{\alpha_2}
        \right], x_{\alpha_3}
      \right], \cdots , x_{\alpha_k}
    \right]
    \\
    &\quad+ c_{\alpha_2} \left[
      \cdots \left[
        \left[
          x_{\alpha_1}, x_{\alpha_2}
        \right], x_{\alpha_3}
      \right], \cdots , x_{\alpha_k}
    \right]
    \\
    &\quad+ c_{\alpha_3} \left[
      \cdots \left[
        \left[
          x_{\alpha_1}, x_{\alpha_2}
        \right], x_{\alpha_3}
      \right], \cdots , x_{\alpha_k}
    \right]
    \\
    &\quad+ ... + c_{\alpha_k} \left[
      \cdots \left[
        \left[
          x_{\alpha_1}, x_{\alpha_2}
        \right], x_{\alpha_3}
      \right], \cdots , x_{\alpha_k}
    \right]
    \\
    &= (c_{\alpha_1} + ... + c_{\alpha_k}) a x_\gamma
    \\
    &= \tilde{c}(\alpha_1 + ... + \alpha_k)ax_\gamma
    \\
    &= \tilde{c}(\gamma)ax_\gamma
  \end{align*}
  whereby $c_\gamma =  \tilde{c}(\gamma)$
  for all $\gamma \in \Phi' \cap \Phi^+$.

  Next let $\gamma \in \Phi' \cap \Phi^-$.
  Consider $[x_\gamma, x_{-\gamma}] \in \f t$, and apply $D'$.
  \begin{align*}
    0
    &= D'([x_\gamma, x_{-\gamma}])
    \\
    &= [D'(x_\gamma),x_{-\gamma}]+[x_\gamma,D'(x_{-\gamma})]
    \\
    &= c_\gamma[x_\gamma,x_{-\gamma}]
    + c_{-\gamma}[x_\gamma,x_{-\gamma}]
    \\
    &= (c_\gamma + c_{-\gamma})[x_\gamma,x_{-\gamma}].
  \end{align*}
  Since $[x_\gamma, x_{-\gamma}] \neq 0$,
  we have $c_\gamma + c_{-\gamma} = 0$ so
  \begin{align*}
    c_\gamma &= -c_{-\gamma} &&
    \\
    &= -\tilde{c}(-\gamma) &&
      \text{since $-\gamma \in \Phi' \cap \Phi^+$}
    \\
    &= \tilde{c}(\gamma) &&
  \end{align*}
  as desired.

  Next, we use the canonical isomorphism $\Psi: \f h^{**} \to \f h$
  \cite[Ch. VII, \S 4]{maclane1967algebra} to produce
  $\Psi(\tilde{c}) = h \in \f h$.
  Notice that for each $\gamma \in \Phi'$ we have the identity
  \begin{equation*}
    \tilde{c}(\gamma) - \gamma ( h ) = 0
  \end{equation*}
  by the definition of the canonical isomorphism.

  The claim is that $D' - \ad h$ annihilates $[\f q,\f q]$.
  Since $[h,\f t]=0$, we need only check that $D' - \ad h$ maps each
  $x_\gamma$ to $0$, for $\gamma\in\Phi'\cap -\Phi'$.
  \[
    (D' - \ad h)(x_\gamma)
    = \tilde{c}(\gamma)x_\gamma - \gamma(h) x_\gamma
    = (\tilde{c}(\gamma) - \gamma(h)) x_\gamma
    = 0
  \]
  verifying the claim.

  \begin{claim}\label{claim:l_adp}
    $D = L + \ad p$ for some $p \in \f q$ and some derivation $L$ which
    maps $\f g_Z \dsum \f c$ to $\f g_Z$ and maps $[\f q, \f q]$ to $0$.
  \end{claim}

  Set $p = x + h$ as above and set $L = D - \ad p = D' - \ad h$.
  Then $D = L + \ad p$ as desired.
  We note that since $L$ is the difference of two derivations,
  $L$ is itself a derivation.
  We know from claim \ref{claim:dadxadh} that $L$ annihilates
  $[\f q, \f q]$.
  We must check that $L$ maps $\f g_Z \dsum \f c$ to $\f g_Z$.

  We have already seen that $\f g_Z$ is the center of $\f q$, and more,
  that a derivation of $\f q$ must stabilize the center of $\f q$.
  What is left to verify claim \ref{claim:l_adp} is to check that $L$
  maps $\f c$ into $\f g_Z$.
  Let $c \in \f c$ be arbitrary.
  We have
  \begin{align*}
    L(c) &= \left( D' - \ad h \right)(c)
    \\
    &= D'(c) - [h,c]
    \\
    &= \underbrace{D'(c) \in \f g_Z
      }_{\text{ by claim \ref{claim:d_adx}}}
  \end{align*}
  verifying claim \ref{claim:l_adp}.

  Since $D$ was arbitrary, we now have that $\Der \f q$ is spanned by
  $\ad \f q$ and the subset of $\Der \f q$ consisting of derivations
  that map $\f g_Z \dsum \f c$ to $\f g_Z$ and $[\f q, \f q]$ to $0$.
  The next three claims establish facts about the relationship between
  these two sets.

  \begin{claim}\label{claim:l_subset_derq}
    $\f L \subseteq \Der \f q$.
  \end{claim}

  $\f L$ is defined as the set of $\K$-linear endomorphisms of $\f q$
  mapping into the center of $\f q$ and mapping $[\f q, \f q]$ to $0$.
  We will show that any such linear map is indeed a derivation of
  $\f q$.
  Suppose $L: \f q \to \f q$ is any $\K$-linear map satisfying
  $L(\f q) \subseteq \f g_Z$ and $L([\f q,\f q]) = 0$.
  Then, for any $x,y \in \f q$ we have
  \[
    [L(x),y] + [x,L(y)] = 0 = L([x,y])
  \]
  so $L$ is a derivation.

  \begin{claim}\label{claim:l_ideal_derq}
    $\f L$ is an ideal of $\Der \f q$.
  \end{claim}

  First we note that $\f L$ is linearly closed:
  Indeed, if $L_1, L_2 \in \f L$, then $L_1 + k L_2$
  maps into $\f g_Z$ and maps $[\f q, \f q]$ to $0$.
  Second, let $L \in \f L$ and $D \in \Der \f q$.
  We must show $[D,L] \in \f L$.
  Recall (from proposition \ref{prop:derivations_stabilize_these})
  that $D(\f g_Z) \subseteq \f g_Z$
  and $D([\f q,\f q]) \subseteq [\f q,\f q]$.
  Let $x \in \f q$. Consider $[D,L](x)$.
  \[
    [D,L](x)
    = D(\underbrace{L(x)}_{\in \f g_Z})
    - L(\underbrace{D(x)}_{\in \f q})
    \in \f g_Z
  \]
  so $[D,L]$ maps $\f q$ into $\f g_Z$.
  Now, let $y \in [\f q,\f q]$ and consider $[D,L](y)$.
  \[
    [D,L](y)
    = D(\underbrace{L(y)}_{0})
    - L(\underbrace{D(y)}_{\in [\f q,\f q]})
    = 0
  \]
  so $[D,L] \in \f L$, verifying the claim.

  \begin{claim}\label{claim:l_disjoint_adq}
    $\f L$ and $\ad \f q$ intersect trivially.
  \end{claim}

  Suppose $D \in \f L \cap \ad \f q$.
  Since $D \in \f L$, $D$ maps $\f q$ into $\f g_Z$.
  Since $D \in \ad \f q$, $D$ maps $\f q$ into $[\f q, \f q]$.
  So, $D$ maps $\f q$ into $\f g_Z \cap [\f q, \f q] = 0$,
  whereby $D = 0$, completing the proof of the theorem.
\end{proof}

As a simple application, we will use proposition \ref{prop:complex_case}
to derive a formula for the dimension of $\Der \f q$ in terms of $\f g$
and $\f q$ and their invariants.

\begin{proposition}\label{cor:dim}
  For $\f q = \K^n \oplus \f q_S$ a parabolic subalgebra
  of a reductive Lie algebra $\f g = \K^n \oplus \f g_S$
  over $\K$, with notation as above,
  the dimension of $\Der \f q$ is given by
  \[
    \dim \left( \Der \f q \right) =
    \left( n + |\Delta| - |\Delta'| \right) n + \dim \f q_S.
  \]
\end{proposition}

\begin{proof}
  The corollary follows from the isomorphism
  $\Der \f q \cong
  \Hom_\K \left( \f g_Z \dsum \f c , \f g_Z \right)
  \oplus \ad \f q$.
  We have $\dim \f c = |\Delta| - |\Delta'|$,
  and $\dim \ad \f q = \dim \f q_S$,
  since $\ad \f q \cong \ad \f q_S \cong \f q_S$.
\end{proof}

\section{The real case}\label{sec:real_case}

In this section, $\f g = \f g_Z \oplus \f g_S$ denotes a reductive Lie
algebra over $\R$ with center $\f g_Z$ and maximal semisimple ideal
$\f g_S$.
$\f q = \f g_Z \oplus \f q_S$ is a parabolic subalgebra of $\f g$,
where $\f q_S = \f q \cap \f g_S$ is a parabolic subalgebra of $\f g_S$.

We begin by proving the central proposition in the context of real Lie
algebras.
The proof will rely heavily on the complexification $\hat{\f g}$ of
$\f g$, to which we will apply proposition \ref{prop:complex_case}.
Afterwards, we consider the restricted root space decomposition of
$\f g$ and expand upon the central theorem.

\begin{theorem}\label{prop:real_case}
  For a parabolic subalgebra $\f q=\f g_Z\oplus\f q_S$ of a reductive
  Lie algebra $\f g=\f g_Z\oplus\f g_S$ over $\R$, the derivation
  algebra $\Der \f q$ decomposes as the sum of ideals
  \[
    \Der \f q = \f L \oplus \ad \f q
  \]
  where $\f L$ consists of all $\R$-linear transformations on $\f q$
  mapping into $\f q_Z$ and mapping $[\f q, \f q]$ to $0$.
\end{theorem}

\begin{proof}
  We may assume without loss of generality that $\f g$ is realized as a
  set of real matrices by Ado's Theorem (theorem \ref{thm:ados_thm}).
  We fix the following notation:
  \begin{itemize}
    \item[]
    $\im$ denotes the imaginary unit;
    \item[]
    $\hat{\f g} = \f g \dsum \im \f g =
    (\f g_Z \dsum \im\f g_Z) \oplus (\f g_S \dsum \im\f g_S)$
    denotes the complexification of $\f g$;
    \item[]
    $\f q = \f g_Z \oplus \f q_S$
    denotes a parabolic subalgebra of $\f g$;
    \item[]
    $\hat{\f q} = \f q \dsum \im\f q =
    (\f g_Z \dsum \im\f g_Z) \oplus (\f q_S \dsum \im\f q_S)$
    is a parabolic subalgebra of $\hat{\f g}$; and
    \item[]
    $\widehat{\f g_Z} = \f g_Z \dsum \im\f g_Z = \hat{\f g}_Z$
    denotes the center of $\hat{\f g}$.
  \end{itemize}

  Given a derivation $D$ of $\f q$, we have a corresponding derivation
  $\hat{D}$ of $\hat{\f q}$ given by
  $\hat{D} (x + \im y) = D(x) + \im D(y)$.
  As a derivation of $\hat{\f q}$, $\hat{D}$ decomposes as
  $\hat{D} = L + \ad (x + \im y)$ with $L$ mapping
  $\hat{\f q}$ into $\hat{\f g}_Z$ and mapping $[\hat{\f q},\hat{\f q}]$
  to $0$ and with $x,y \in \f q$. Note that $L$ sends $[\f q, \f q]$ to
  $0$, since $[\f q,\f q] \subseteq [\hat{\f q},\hat{\f q}]$.

  Let $u \in \f q$.
  $\hat{D}$ stabilizes $\f q$, so we have
  \[
    D(u) = \hat{D}(u) =
    L(u) + \ad (x+\im y) (u) = L(u) + [x,u] + \im [y,u]
    \in \f q
  \]
  Now, $L(u) \in \hat{\f g_Z} = \f g_Z + \im \f g_Z$,
  so we may write $L(u) = v_1 + iv_2$ with $v_1, v_2 \in \f g_Z$.
  Then
  \[
    D(u) =
    v_1 + iv_2 + [x,u] + \im [y,u] =
    (v_1 + [x,u]) + \im (v_2 + [y,u])
    \in \f q
  \]
  so $v_2 + [y,u] = 0$, and by the direct-sum decomposition
  $\f q = \f g_Z + [\f q, \f q]$, $v_2 = 0$ and $[y,u] = 0$.
  In particular, we have $L(u) = v_1$, so $L$ maps $\f q$ into $\f g_Z$.
  Furthermore, since $u$ was arbitrary and $[y,u] = 0$,
  we have $y \in \f g_Z$. Since $y \in \f g_Z$,
  we have for any arbitrary $z = u + \im v \in \f g$,
  \[
    \ad (x+\im y)(z)
    = \ad x (z) + \im \ad y (z)
    = \ad x(z)+\im [y,u]-[y,v]
    = \ad x(z),
  \]
  thus we have
  $\ad (x+\im y) = \ad x$.

  We now have $D = L|_{\f q} + \ad x$ with $x \in \f q$ and $L|_{\f q}$
  an $\R$-linear transformation mapping $\f q$ to $\f g_Z  (= \f q_Z)$
  and $[\f q, \f q]$ to $0$, as desired.
  We have left to check that arbitrary $\R$-linear maps sending $\f q$
  to $\f g_Z$ and $[\f q,\f q]$ to $0$ are derivations, that $\f L$ as
  described is an ideal of $\f q$, and that $\f L$ and $\ad \f q$
  intersect trivially, the proofs of which are identical to the proofs
  given of claims \ref{claim:l_subset_derq}, \ref{claim:l_ideal_derq},
  and \ref{claim:l_disjoint_adq} of theorem \ref{prop:complex_case},
  respectively.
\end{proof}

We now examine the relationship between the direct sum decomposition of
$\Der \f q$ and the restricted root space decomposition of $\f g$.
Given a parabolic subalgebra $\f q = \f g_Z \oplus \f q_S$ of a
reductive real Lie algebra $\f g = \f g_Z \oplus \f g_S$, we may choose
a restricted root space decomposition of $\f g$ that is compatible with
$\f q$ in the sense that $\f q$ is a standard parabolic subalgebra of
$\f g$.
We may then decompose $\f q$ into the sum of
$\f q = \f g_Z \dsum \f c \dsum [\f q,\f q]$ where $\f c$ is an
appropriately-chosen complimentary subspace, similar to the
algebraically closed case.
To achieve this decomposition, we rely on Langland's decomposition of
$\f q_S$, described in chapter \ref{sec:preliminaries}.
We fix the following notation pertaining to the restricted root space
decomposition of $\f g$:
\begin{itemize}
  \item[]
  $\f g = \f g_Z \dsum \f a \dsum \f m \dsum
  \sum_{\alpha \in \Phi} \f g_\alpha$;
  \item[]
  $\Delta$ a base of $\Phi$;
  \item[]
  $\Delta' \subseteq \Delta$ corresponding to $\f q_S$;
  \item[]
  $\Phi' = \Phi^+ \cup \left( \Phi \cap \Span \Delta' \right)$; and
  \item[]
  $\f q = \f g_Z \dsum \underbrace{\f a \dsum \f m \dsum
  \sum_{\gamma \in \Phi'} \f g_\gamma}_{\f q_S}$.
\end{itemize}
Write Langland's decomposition of $\f q_S$:
\begin{itemize}
  \item[]
  $\f l = \f a \dsum \f m \dsum
  \sum_{\gamma \in \Phi' \cap -\Phi'} \f g_\gamma$ and
  \item[]
  $\f n = \sum_{\gamma \in \Phi' \setminus -\Phi'} g_\gamma$
\end{itemize}
so that $\f q_S = \f l \ltimes \f n$ with $\f l$ reductive and $\f n$
nilpotent.
Write
\begin{itemize}
  \item[]
  $\f c$ for the center of $\f l$ and
  \item[]
  $\f l_S$ for the unique semisimple ideal of $\f l$
\end{itemize}
so that $\f l = \f c \oplus \f l_S$.

\begin{claim*}
  $[\f q, \f q] = \f l_S \dsum \f n$
\end{claim*}

\begin{proof}
  Let $x, y \in \f q$.
  We must show $[x,y] \in \f l_S \dsum \f n$.
  Without loss of generality, we may assume $x,y \in \f q_S$,
  since their projections onto $\f g_Z$ are lost upon applying bracket.

  Write $x = x_{\f l} + x_{\f n}$ and $y = y_{\f l} + y_{\f n}$ with
  $x_{\f l},y_{\f l} \in \f l$ and $x_{\f n},y_{\f n} \in \f n$.
  Then
  \begin{align*}
    [x,y] &= [ x_{\f l} + x_{\f n} , y_{\f l} + y_{\f n} ] \\
    &= \underbrace{[ x_{\f l} , y_{\f l} ]}_{\in \f l_S}
      + \underbrace{[ x_{\f l} , y_{\f n} ] + [ x_{\f n} , y_{\f l} ]
      + [ x_{\f n} , y_{\f n} ]}_{\in \f n}
    \in \f l_S \dsum \f n
  \end{align*}
  since $\f l$ is reductive and $\f n$ is an ideal of $\f q_S$.
\end{proof}

Thus we arrive at the desired decomposition,
\begin{align*}
  \f q &= \f g_Z \dsum \underbrace{\f l \dsum \f n}_{\f g_S} \\
  &= \f g_Z \dsum \f c \dsum \f l_S \dsum \f n \\
  &= \f g_Z \dsum \f c \dsum [\f q, \f q]
  \text{.}
\end{align*}

\begin{theorem}\label{prop:real_decomp}
  For any root system $\Phi$ with respect to which $\f q$ is a standard
  parabolic subalgebra, $\f q$ decomposes as
  $\f q=\f g_Z \dsum \f c \dsum [\f q,\f q]$
  (where $\f c$ is the center of the Levi factor $\f l$ of $\f q$)
  and the ideal $\f L$ of $\Der \f q$ consists of all $\R$-linear
  transformation on $\f q$ that map $\f g_Z \dsum \f c$ to $\f g_Z$ and
  map $[\f q,\f q]$ to $0$, whereby
  \[
    \Der \f q \cong
    \Hom_\R \left( \f g_Z  \dsum  \f c , \f g_Z \right)
    \oplus \ad \f q
    \text{.}
  \]
\end{theorem}

\begin{proof}
  The proof is essentially done.
  The majority is merely the description of the decomposition of $\f q$,
  already done above.
  We have left to show only that
  $\f L \cong \Hom_\R \left( \f g_Z + \f c , \f g_Z \right)$,
  which is obvious in light of the decomposition
  $\f q = \f g_Z  \dsum  \f c  \dsum  [\f q, \f q]$.
\end{proof}

\section{Further remarks and corollaries}\label{sec:corollaries}

The following three propositions represent extremal cases of the central
theorem.
Proposition \ref{cor:semisimple} applies to arbitrary parabolic
subalgebras of a semisimple Lie algebra (ie, the case $\f g_Z = 0$).
Proposition \ref{cor:borel} and \ref{cor:total} apply specifically to
minimal parabolic subalgebras (ie, Borel subalgebras) and maximal
parabolic subalgebras (ie, the entire Lie algebra $\f g$), respectively.

\begin{proposition}\label{cor:semisimple}
  For a parabolic subalgebra $\f q$ of a semisimple Lie algebra $\f g$
  over the field $\K$ or over $\R$, the derivation algebra $\Der \f q$
  satisfies
  \[
    \Der \f q = \ad \f q
    \text{.}
  \]
\end{proposition}

\begin{proof}
  By the central theorem, $\Der \f q = \f L \oplus \ad \f q$, and since
  $\f g_Z=0$, we have $\f L=0$.
\end{proof}

Proposition \ref{cor:semisimple} was proven for Borel subalgebras of
semisimple Lie algebras over an arbitrary field by Leger and Luks in
\cite{leger1972cohomology}.
Tolpygo found the same result for parabolic subalgebras of complex Lie
algebras \cite{tolpygo1972cohomologies}.
Our results show that the same is true for real Lie algebras.

\begin{proposition}\label{cor:borel}
  For a Borel subalgebra
  $\f b = \f g_Z \dsum \f g_0 \dsum
  \sum_{\alpha \in \Phi^+} \f g_\alpha$
  of the reductive Lie algebra $\f g = \f g_Z \oplus \f g_S$ over $\K$
  or over $\R$, the derivation algebra $\Der \f b$ satisfies
  \[
    \Der \f b \cong
    \Hom_{\F} \left( \f g_Z \dsum \left(\f g_0\right)_Z,
    \f g_Z \right) \oplus \ad \f b
    \text{.}
  \]
\end{proposition}

\begin{proof}
  Write $\f b_S = \f g_0 \dsum \sum_{\alpha \in \Phi^+} \f g_\alpha$.
  Since $\sum_{\alpha \in \Phi^+} \f g_\alpha$ is clearly the nilpotent
  radical of $\f b_S$, the Levi factor $\f l = \f g_0$. Applying the
  central theorem gives the result.
\end{proof}

Farnsteiner proved proposition \ref{cor:borel} for semisimple algebras
over an algebraically closed field in \cite{farnsteiner1988derivations}.
As with proposition \ref{cor:semisimple}, our results establish this
fact for real Lie algebras, as well as establishing the result for
reductive algebras.

\begin{proposition}\label{cor:total}
  For a reductive Lie algebra $\f g = \f g_Z \oplus \f g_S$ over $\K$ or
  over $\R$, the derivation algebra $\Der \f g$ satisfies
  \[
    \Der \f g \cong \f{gl}(\f g_Z) \oplus \ad \f g
    \text{.}
  \]
\end{proposition}

\begin{proof}
  $\f g_S$ is its own Levi factor. Being semisimple, the center of
  $\f g_S$ is trivial, so $\f L$ consists of linear maps stabilizing
  $\f g_Z$ and sending $\f g_S = [\f g, \f g]$ to $0$, which are exactly
  the linear maps on $\f g_Z$ direct sum with the zero map on $\f g_S$.
\end{proof}

The next corollary provides a high-level abstract description of
$\Der \f q$ useful for dim\-en\-sion-count\-ing arguments.
It is also satisfying on a theoretical level, since it relies on simple
constructions that can be carried out on any Lie algebra, suggesting
that the result here for parabolic subalgebras of reductive Lie algebras
might be generalized to a larger classes of Lie algebras.

\begin{proposition}\label{cor:abstract_decomposition}
  For a parabolic subalgebra $\f q$ of a reductive Lie algebra $\f g$
  over $\K$ or over $\R$, we have
  \[
    \Der \f q \cong
    \Hom \left( \f q / [\f q, \f q], \f q_Z \right)
    \oplus \left( \f q / \f q_Z \right).
  \]
\end{proposition}

\begin{proof}
  Recall that $\f q / [\f q, \f q]$ is the minimal abelian quotient of
  $\f q$.
  Since $\f g_Z \dsum \f c$ is abelian and since
  $\f q = \f g_Z \dsum \f c \dsum [\f q, \f q]$,
  we have $\f g_Z \dsum \f c \cong \f q / [\f q, \f q]$.
  Also, $\f g_Z = \f q_Z$, and $\ad \f q \cong \f q / \f q_Z$,
  thus the corollary.
\end{proof}

Our work largely follows the results of Leger and Luks and Tolpygo.
Leger's and Luks's results imply that all derivations of a Borel
subalgebra of a simple Lie algebra are inner (over any field with
characteristic not 2) \cite{leger1972cohomology}, and similarly
Tolpygo's results (applicable specifically over the complex field)
imply that all derivations of a parabolic subalgebra of a semisimple
Lie algebra are inner \cite{tolpygo1972cohomologies}.
The full results in these papers are somewhat more general and stated
in the language of cohomology:
The authors prove that all cohomology group $H^n(\f g, \f g)$ are
trivial for their respective classes of Lie algebras $\f g$ under
consideration \cite{leger1972cohomology, tolpygo1972cohomologies}.

For instance, the first cohomology group $H^1(\f g; \f g)$ of a Lie
algebra $\f g$ satisfies the isomorphism
\[
  H^1(\f g; \f g) \cong \Der \f g / \ad \f g.
\]
From this isomorphism, it follows that $H^1(\f g; \f g) = 0$ implies
that all derivations of $\f g$ are inner.
A further application of cohomology is to extensions of a Lie algebra
$\f b$ by an ideal $\f a$. We have the isomorphism
\[
  H^2(\f b; \f a) \cong \mathop{\mathrm{Ext}}(\f b; \f a)
\]
so the second cohomology group $H^2(\f b; \f a)$ parametrizes the
extensions of $\f b$ by $\f a$.

In light of these two isomorphisms, the language and methods of
cohomology provide a strong framework for discovering structural
properties of $\Der \f g$ as they relate to properties of $\f g$.
Our results on derivations apply to reductive Lie algebras,
direct extensions of semisimple Lie algebras by an abelian Lie algebra.
A consideration of $H^2$ might be employed to study the derivations of
general extensions of Lie algebras.

A second vehicle for future research that we will discuss deals with
the abstract form of the decomposition of $\Der \f q$ given in
proposition \ref{cor:abstract_decomposition}.
If we denote by $\f g$ a parabolic subalgebra, we have that the
derivation algebra $\Der \f g$ decomposes as
\begin{equation}\label{eq:abstract_decomposition}
\Der \f g \cong
\Hom \left( \f g / [\f g, \f g], \f g_Z \right) \oplus \ad \f g
\text{.}
\end{equation}
The constructions necessary to express isomorphism
\ref{eq:abstract_decomposition} are completely general,
motivating the following question:
for which Lie algebras $\f g$ does isomorphism
\ref{eq:abstract_decomposition} hold?

We remind the reader that a Lie algebra $\f g$ is called \emph{complete}
if $\f g_Z = 0$ and $\f g$ has only inner derivations.
Analogously, we propose the following definition: a Lie algebra $\f g$
is \emph{almost complete} if \eqref{eq:abstract_decomposition}
holds.
As an area for future investigation we may wish to
characterize the class of almost complete Lie algebras.

\bibliographystyle{plain}
\bibliography{article}

\end{document}